\documentclass[a4paper,11pt,reqno]{amsart}

\usepackage{standalone}
\usepackage{lmodern}
\usepackage{amsmath,amssymb,amsthm}
\usepackage{mathrsfs}
\usepackage{mathtools}
\usepackage{ifthen}
\usepackage{graphicx}
\usepackage[T1]{fontenc}
\usepackage{tikz}
\usepackage{comment}
\usepackage{todonotes}
\usepackage{xcolor}
\usepackage{enumitem}
\usepackage{esint}

\usetikzlibrary{arrows}
\usetikzlibrary{shapes}
\usetikzlibrary{patterns}
\usetikzlibrary{decorations}

\usepackage[
    colorlinks=true,
    linkcolor=red,
    citecolor=blue,
    urlcolor=blue
]{hyperref}

\usepackage{cleveref}

{\scriptsize\ttfamily\color{black!35}}

\numberwithin{equation}{section}

\newtheorem{thm}{Theorem}[section]
\newtheorem*{thm*}{Theorem}

\newtheorem{lem}[thm]{Lemma}

\newtheorem{cor}[thm]{Corollary}

\theoremstyle{definition}

\newtheorem{defn}[thm]{Definition}

\newtheorem{prob}[thm]{Problem}
\newtheorem{rem}[thm]{Remark}

\newcounter{minutes}
\divide\time by 60

\newcounter{hours}
\multiply\time by 60
\addtocounter{minutes}{-\time}

\newcounter{own}[section]

\renewcommand{\theown}{\thesection.\arabic{own}}

\newenvironment{pf}[1][]
{
    \vskip 3mm
    \noindent
    \ifthenelse{\equal{#1}{}}
    {{\slshape Proof. }}
    {{\slshape #1.} }
}
{
    \qed
    \bigskip
}

\newcounter{alphabet}

\def\be{\begin{equation}}
\def\ee{\end{equation}}

\newcommand{\bee}{\begin{enumerate}}
\newcommand{\eee}{\end{enumerate}}

\newcommand{\blem}{\begin{lem}}
\newcommand{\elem}{\end{lem}}

\newcommand{\bthm}{\begin{thm}}
\newcommand{\ethm}{\end{thm}}

\newcommand{\bcor}{\begin{cor}}
\newcommand{\ecor}{\end{cor}}

\newcommand{\beg}{\begin{examp}}
\newcommand{\eeg}{\end{examp}}

\newcommand{\begs}{\begin{examples}}
\newcommand{\eegs}{\end{examples}}

\newcommand{\bdefe}{\begin{defin}}
\newcommand{\edefe}{\end{defin}}

\newcommand{\bprob}{\begin{prob}}
\newcommand{\eprob}{\end{prob}}

\newcommand{\bei}{\begin{itemize}}
\newcommand{\eei}{\end{itemize}}

\newcommand{\D}{\mathbb D}
\newcommand{\HH}{\mathcal H}

\newcommand{\Mod}{\operatorname{Mod}}
\newcommand{\dist}{\operatorname{dist}}

\title{Characterizations for quasiconformal weighted Bergman spaces}
\author{Pekka Koskela, Chengxi Li, Abhishek Pandey}

\address{P.\ Koskela: Department of Mathematics and Statistics, University of Jyv\"askyl\"a, P.O. Box 35, FI-40014, University of Jyv\"askyl\"a, Finland. {\tt pekka.j.koskela@jyu.fi }}
\address{C.\ Li: Department of Mathematics and Statistics, University of Jyv\"askyl\"a, P.O. Box 35, FI-40014, University of Jyv\"askyl\"a, Finland. {\tt cheli@jyu.fi }} 
\address{A.\ Pandey: Department of Mathematics and Statistics, University of Jyv\"askyl\"a, P.O. Box 35, FI-40014, University of Jyv\"askyl\"a, Finland. {\tt pandey.a.pandey@jyu.fi}}

\thanks{P.K. was supported by the Academy of Finland
via Centre of Excellence in Analysis and Dynamics Research, project number 323960.
C.L. was supported by the Research Council of Finland, project numbers 364210 and 360505, and China Scholarship Council Fellowship, project number 202306200022. A.P. was supported by the Research Council of Finland, project number 360505.}

\begin{document}

\begin{abstract}
We establish characterizations of planar quasiconformal mappings in weighted Bergman spaces $A_\alpha^p(\mathbb D)$. For $p>0$ and $\alpha>-1$, our characterization involves the convergence of certain integrals involving superlevel set areas, the Hausdorff measure of boundary superlevel sets, and the averaged derivative. These include quasiconformal analogues of characterizations established by Karafyllia and Karamanlis for conformal mappings. Moreover, for $-1<\alpha\leq 0$, we obtain an additional equivalent characterization in terms of the Hausdorff measure of level sets. The power $\alpha$ is sharp: for every $\alpha>0$, we construct a bounded
$K_\alpha$-quasiconformal mapping for which the weighted level-set
length integral with exponent $\alpha+2$ diverges.
\end{abstract}
	
	\maketitle
	\pagestyle{myheadings}
	\markboth{Characterizations for QC weighted Bergman spaces}{Pekka Koskela, Chengxi Li, Abhishek Pandey}
	
	\section{Introduction}
    \subsection{Background}
    A classical problem in geometric function theory is to characterize mappings that belong to Hardy or weighted Bergman spaces in terms of geometric properties.
    In this paper, we study characterizations for planar quasiconformal mappings in weighted Bergman spaces.

    For $p>0$ and $\alpha>-1$, the weighted Bergman space $A^p_{\alpha}(\mathbb{D})$ consists of holomorphic functions $f$ on the unit disk $\mathbb{D}$ satisfying
    \begin{equation}\label{eq:weightedBergman}
        \|f\|_{A^p_{\alpha}(\mathbb{D})}^p
	:=
	\int_{\mathbb{D}} |f(z)|^p(1-|z|^2)^{\alpha}\,dA(z)<\infty,
    \end{equation}
	where $dA$ is Lebesgue area measure. The related Hardy space $H^p(\D)$ is defined by
    \begin{equation}\label{eq:Hardy}
        \|f\|_{H^p(\mathbb{D})}^p
	:=
	\sup_{0<r<1}\int_0^{2\pi}|f(re^{it})|^p\,dt<\infty.
    \end{equation}

A quasiconformal mapping of the unit disk belongs to $A^p_{\alpha}(\mathbb{D})$ or $H^p(\mathbb{D})$ if \eqref{eq:weightedBergman}
or \eqref{eq:Hardy} holds, respectively. Since $(1-|z|^2)^\alpha\sim (1-|z|)^\alpha$ for every $\alpha>-1$ and $z\in\D$, we may replace $(1-|z|^2)^\alpha$ with the weight $(1-|z|)^\alpha$.

The study of integral means and maximum moduli has its roots in the work of Hardy and Littlewood~\cite{HL32}. For conformal mappings, maximum-modulus characterization for Bergman spaces were established by Baernstein, Girela, and Pel\'aez~\cite{BGP04}, and corresponding weighted analogue were studied by P\'erez-Gonz\'alez and R\"atty\"a~\cite{PR08}.

Another important characterization for conformal mappings in such spaces are in terms of the geometric properties of level sets. Following Poggi-Corradini's harmonic-measure approach to Hardy spaces~\cite{PC97}, Karafyllia~\cite{Kar19} obtained a
characterization in terms of hyperbolic distance. Later, Betsakos, Karafyllia, and Karamanlis~\cite{BKK21} extended the latter criterion to weighted Bergman spaces. Karafyllia and Karamanlis~\cite{KK}     subsequently established some interesting geometric characterizations for weighted Bergman spaces. These motivate the quasiconformal characterizations in this paper.

For quasiconformal mappings, the Hardy theory was initiated by
Zinsmeister~\cite{Zin86} and developed systematically by Astala and
the first author~\cite{AK}. Later on, the first author with Benedict and Li~\cite{BKL} developed a weighted Hardy theory connecting weighted
maximum modulus integrals with boundary maximal functions and averaged
derivatives.
    
\subsection{Main results} 
  Let $f:\D\to\Omega\subsetneq\mathbb C$ be an
orientation-preserving quasiconformal homeomorphism. Set
  $$M(r):=\max_{|z|=r}|f(z)|,\quad \text{for}\,\,\,0\le r<1,$$
and for $r>0$, define
  $$F_r:=\{z\in\mathbb D:|f(z)|=r\},\quad
  U_r:=\{z\in\mathbb D:|f(z)|>r\}.$$
The finite radial limit $f(\omega)=\lim_{r\to1^-}f(r\omega)$ exists
outside a null set $N_f\subset\partial\mathbb D$ of
$\HH^1$-measure zero; see \cite[p.~21]{AK}. We define
\begin{equation}\label{defofEr}
  E_r:=\{\omega\in\partial\mathbb D\setminus N_f:|f(\omega)|>r\},
  \quad m(r):=\mathcal H^1(E_r).
\end{equation}
Here $\mathcal H^1$ denotes one-dimensional Hausdorff measure and
$|U_r|$ denotes the Euclidean area. See Figure 1 for an illustration of the sets $U_r$, $E_r$, and $F_r$,
and their images under a quasiconformal mapping $f:\D\to\Omega$.
Note that in the quasiconformal setting, the level set $F_r$
need not be rectifiable.

Finally, let
  $$a_f(z):=\exp\left(\frac{1}{2|B_z|}
                    \int_{B_z}\log J_f(w)\,dA(w)\right),
  \quad B_z:=B\left(z,\frac{1-|z|}{2}\right),$$
where $J_f$ is the Jacobian determinant of $f$. For conformal $f$, the above $a_f(z)$ coincides with $|f'(z)|$. 

\begin{comment}
Fix an arc $I\subset \partial \D$, we define
		its Carleson tent by
	\begin{equation}\label{eq:Carleson-tent}
		T(I):=
		\left\{
		\rho\omega:
		\omega\in I,\ 0\leq\rho<1,\ 
		1-\rho\leq\HH^1(I)
		\right\}.
	\end{equation}

    We use the hyperbolic distance
		$$
		d_\D(0,z)=\log\frac{1+|z|}{1-|z|}
		$$
		and set
		$$
		d_\D(0,F_r):=\inf_{z\in F_r}d_\D(0,z),
		$$
		with the convention that $d_\D(0,\emptyset)=\infty$. For $t>0$, let $G_r$ be the connected component of
		$\D\setminus F_r$ containing $0$, and set $A_r:=\partial G_r\cap\D=\partial G_r\cap F_r$. We define
		$$
		\omega_\D(0,F_r):=\omega_{G_r}(0,A_r).
		$$
        where $\omega_{G_r}(0,A_r)$ denotes harmonic measure in $G_r$ with pole at 0.
		If $F_r=\emptyset$, we we adopt the conventions 
        $$d_\D(0,F_r):=\infty\quad\text{and}\quad\omega_\D(0,F_r):=0.$$
        Moreover,
		$$
		\omega_\D(0,E_r)=\frac{m(t)}{2\pi}.
		$$
\end{comment}

\begin{figure}[htbp]
\centering
\begingroup

% Precomputed snowflake boundary.
\def\SFSnowflake{
  (-1.2025,-1.2025) -- (-1.09939,-1.2025) -- (-1.05519,-1.29566) -- (-1.011,-1.2025)
  -- (-0.90789,-1.2025) -- (-0.8637,-1.29566) -- (-0.92893,-1.37552) -- (-0.82582,-1.37552)
  -- (-0.78163,-1.46868) -- (-0.73743,-1.37552) -- (-0.63432,-1.37552) -- (-0.69955,-1.29566)
  -- (-0.65536,-1.2025) -- (-0.55225,-1.2025) -- (-0.50806,-1.29566) -- (-0.46386,-1.2025)
  -- (-0.36075,-1.2025) -- (-0.31656,-1.29566) -- (-0.38179,-1.37552) -- (-0.27868,-1.37552)
  -- (-0.23449,-1.46868) -- (-0.29972,-1.54854) -- (-0.39983,-1.52382) -- (-0.35564,-1.61699)
  -- (-0.42087,-1.69684) -- (-0.31776,-1.69684) -- (-0.27357,-1.79001) -- (-0.22938,-1.69684)
  -- (-0.12626,-1.69684) -- (-0.08207,-1.79001) -- (-0.14731,-1.86986) -- (-0.04419,-1.86986)
  -- (0,-1.96303) -- (0.04419,-1.86986) -- (0.14731,-1.86986) -- (0.08207,-1.79001)
  -- (0.12626,-1.69684) -- (0.22938,-1.69684) -- (0.27357,-1.79001) -- (0.31776,-1.69684)
  -- (0.42087,-1.69684) -- (0.35564,-1.61699) -- (0.39983,-1.52382) -- (0.29972,-1.54854)
  -- (0.23449,-1.46868) -- (0.27868,-1.37552) -- (0.38179,-1.37552) -- (0.31656,-1.29566)
  -- (0.36075,-1.2025) -- (0.46386,-1.2025) -- (0.50806,-1.29566) -- (0.55225,-1.2025)
  -- (0.65536,-1.2025) -- (0.69955,-1.29566) -- (0.63432,-1.37552) -- (0.73743,-1.37552)
  -- (0.78163,-1.46868) -- (0.82582,-1.37552) -- (0.92893,-1.37552) -- (0.8637,-1.29566)
  -- (0.90789,-1.2025) -- (1.011,-1.2025) -- (1.05519,-1.29566) -- (1.09939,-1.2025)
  -- (1.2025,-1.2025) -- (1.2025,-1.09939) -- (1.29566,-1.05519) -- (1.2025,-1.011)
  -- (1.2025,-0.90789) -- (1.29566,-0.8637) -- (1.37552,-0.92893) -- (1.37552,-0.82582)
  -- (1.46868,-0.78163) -- (1.37552,-0.73743) -- (1.37552,-0.63432) -- (1.29566,-0.69955)
  -- (1.2025,-0.65536) -- (1.2025,-0.55225) -- (1.29566,-0.50806) -- (1.2025,-0.46386)
  -- (1.2025,-0.36075) -- (1.29566,-0.31656) -- (1.37552,-0.38179) -- (1.37552,-0.27868)
  -- (1.46868,-0.23449) -- (1.54854,-0.29972) -- (1.52382,-0.39983) -- (1.61699,-0.35564)
  -- (1.69684,-0.42087) -- (1.69684,-0.31776) -- (1.79001,-0.27357) -- (1.69684,-0.22938)
  -- (1.69684,-0.12626) -- (1.79001,-0.08207) -- (1.86986,-0.14731) -- (1.86986,-0.04419)
  -- (1.96303,0) -- (1.86986,0.04419) -- (1.86986,0.14731) -- (1.79001,0.08207)
  -- (1.69684,0.12626) -- (1.69684,0.22938) -- (1.79001,0.27357) -- (1.69684,0.31776)
  -- (1.69684,0.42087) -- (1.61699,0.35564) -- (1.52382,0.39983) -- (1.54854,0.29972)
  -- (1.46868,0.23449) -- (1.37552,0.27868) -- (1.37552,0.38179) -- (1.29566,0.31656)
  -- (1.2025,0.36075) -- (1.2025,0.46386) -- (1.29566,0.50806) -- (1.2025,0.55225)
  -- (1.2025,0.65536) -- (1.29566,0.69955) -- (1.37552,0.63432) -- (1.37552,0.73743)
  -- (1.46868,0.78163) -- (1.37552,0.82582) -- (1.37552,0.92893) -- (1.29566,0.8637)
  -- (1.2025,0.90789) -- (1.2025,1.011) -- (1.29566,1.05519) -- (1.2025,1.09939)
  -- (1.2025,1.2025) -- (1.09939,1.2025) -- (1.05519,1.29566) -- (1.011,1.2025)
  -- (0.90789,1.2025) -- (0.8637,1.29566) -- (0.92893,1.37552) -- (0.82582,1.37552)
  -- (0.78163,1.46868) -- (0.73743,1.37552) -- (0.63432,1.37552) -- (0.69955,1.29566)
  -- (0.65536,1.2025) -- (0.55225,1.2025) -- (0.50806,1.29566) -- (0.46386,1.2025)
  -- (0.36075,1.2025) -- (0.31656,1.29566) -- (0.38179,1.37552) -- (0.27868,1.37552)
  -- (0.23449,1.46868) -- (0.29972,1.54854) -- (0.39983,1.52382) -- (0.35564,1.61699)
  -- (0.42087,1.69684) -- (0.31776,1.69684) -- (0.27357,1.79001) -- (0.22938,1.69684)
  -- (0.12626,1.69684) -- (0.08207,1.79001) -- (0.14731,1.86986) -- (0.04419,1.86986)
  -- (0,1.96303) -- (-0.04419,1.86986) -- (-0.14731,1.86986) -- (-0.08207,1.79001)
  -- (-0.12626,1.69684) -- (-0.22938,1.69684) -- (-0.27357,1.79001) -- (-0.31776,1.69684)
  -- (-0.42087,1.69684) -- (-0.35564,1.61699) -- (-0.39983,1.52382) -- (-0.29972,1.54854)
  -- (-0.23449,1.46868) -- (-0.27868,1.37552) -- (-0.38179,1.37552) -- (-0.31656,1.29566)
  -- (-0.36075,1.2025) -- (-0.46386,1.2025) -- (-0.50806,1.29566) -- (-0.55225,1.2025)
  -- (-0.65536,1.2025) -- (-0.69955,1.29566) -- (-0.63432,1.37552) -- (-0.73743,1.37552)
  -- (-0.78163,1.46868) -- (-0.82582,1.37552) -- (-0.92893,1.37552) -- (-0.8637,1.29566)
  -- (-0.90789,1.2025) -- (-1.011,1.2025) -- (-1.05519,1.29566) -- (-1.09939,1.2025)
  -- (-1.2025,1.2025) -- (-1.2025,1.09939) -- (-1.29566,1.05519) -- (-1.2025,1.011)
  -- (-1.2025,0.90789) -- (-1.29566,0.8637) -- (-1.37552,0.92893) -- (-1.37552,0.82582)
  -- (-1.46868,0.78163) -- (-1.37552,0.73743) -- (-1.37552,0.63432) -- (-1.29566,0.69955)
  -- (-1.2025,0.65536) -- (-1.2025,0.55225) -- (-1.29566,0.50806) -- (-1.2025,0.46386)
  -- (-1.2025,0.36075) -- (-1.29566,0.31656) -- (-1.37552,0.38179) -- (-1.37552,0.27868)
  -- (-1.46868,0.23449) -- (-1.54854,0.29972) -- (-1.52382,0.39983) -- (-1.61699,0.35564)
  -- (-1.69684,0.42087) -- (-1.69684,0.31776) -- (-1.79001,0.27357) -- (-1.69684,0.22938)
  -- (-1.69684,0.12626) -- (-1.79001,0.08207) -- (-1.86986,0.14731) -- (-1.86986,0.04419)
  -- (-1.96303,0) -- (-1.86986,-0.04419) -- (-1.86986,-0.14731) -- (-1.79001,-0.08207)
  -- (-1.69684,-0.12626) -- (-1.69684,-0.22938) -- (-1.79001,-0.27357) -- (-1.69684,-0.31776)
  -- (-1.69684,-0.42087) -- (-1.61699,-0.35564) -- (-1.52382,-0.39983) -- (-1.54854,-0.29972)
  -- (-1.46868,-0.23449) -- (-1.37552,-0.27868) -- (-1.37552,-0.38179) -- (-1.29566,-0.31656)
  -- (-1.2025,-0.36075) -- (-1.2025,-0.46386) -- (-1.29566,-0.50806) -- (-1.2025,-0.55225)
  -- (-1.2025,-0.65536) -- (-1.29566,-0.69955) -- (-1.37552,-0.63432) -- (-1.37552,-0.73743)
  -- (-1.46868,-0.78163) -- (-1.37552,-0.82582) -- (-1.37552,-0.92893) -- (-1.29566,-0.8637)
  -- (-1.2025,-0.90789) -- (-1.2025,-1.011) -- (-1.29566,-1.05519) -- (-1.2025,-1.09939)
  -- cycle
}

% Smooth boundary of the image domain.
\def\SFOmega{
  (2.380211,0.000000)
  .. controls (2.394490,0.155784) and (2.331234,0.313532) .. (2.207234,0.439046)
  .. controls (2.083234,0.564560) and (1.901743,0.655961) .. (1.718416,0.711791)
  .. controls (1.535089,0.767621) and (1.353268,0.789382) .. (1.213079,0.810553)
  .. controls (1.072889,0.831725) and (0.974657,0.854920) .. (0.914726,0.914726)
  .. controls (0.854796,0.974533) and (0.831352,1.072466) .. (0.809921,1.212133)
  .. controls (0.788491,1.351800) and (0.766713,1.532896) .. (0.711791,1.718416)
  .. controls (0.656869,1.903935) and (0.567341,2.091229) .. (0.442529,2.224741)
  .. controls (0.317716,2.358254) and (0.158942,2.434559) .. (0.000000,2.428456)
  .. controls (-0.158942,2.422352) and (-0.314704,2.333837) .. (-0.437134,2.197620)
  .. controls (-0.559563,2.061403) and (-0.647699,1.881796) .. (-0.711791,1.718416)
  .. controls (-0.775884,1.555035) and (-0.819011,1.411082) .. (-0.874078,1.308150)
  .. controls (-0.929144,1.205217) and (-0.998593,1.141353) .. (-1.090607,1.090607)
  .. controls (-1.182621,1.039861) and (-1.295454,0.998837) .. (-1.407426,0.940412)
  .. controls (-1.519397,0.881987) and (-1.628214,0.806195) .. (-1.718416,0.711791)
  .. controls (-1.808618,0.617387) and (-1.881325,0.505992) .. (-1.936722,0.385238)
  .. controls (-1.992120,0.264484) and (-2.031672,0.133932) .. (-2.046329,0.000000)
  .. controls (-2.060986,-0.133932) and (-2.048740,-0.271789) .. (-1.994888,-0.396808)
  .. controls (-1.941036,-0.521827) and (-1.843581,-0.631869) .. (-1.718416,-0.711791)
  .. controls (-1.593251,-0.791713) and (-1.442992,-0.840101) .. (-1.310593,-0.875710)
  .. controls (-1.178194,-0.911319) and (-1.067343,-0.937266) .. (-0.993716,-0.993716)
  .. controls (-0.920089,-1.050166) and (-0.882361,-1.140680) .. (-0.848241,-1.269482)
  .. controls (-0.814121,-1.398284) and (-0.779379,-1.563474) .. (-0.711791,-1.718416)
  .. controls (-0.644203,-1.873358) and (-0.543179,-2.013399) .. (-0.417148,-2.097146)
  .. controls (-0.291118,-2.180893) and (-0.143008,-2.207837) .. (0.000000,-2.185005)
  .. controls (0.143008,-2.162172) and (0.279735,-2.092809) .. (0.399204,-2.006932)
  .. controls (0.518672,-1.921055) and (0.622053,-1.819883) .. (0.711791,-1.718416)
  .. controls (0.801529,-1.616949) and (0.877635,-1.513760) .. (0.944546,-1.413612)
  .. controls (1.011457,-1.313464) and (1.068449,-1.216425) .. (1.130454,-1.130454)
  .. controls (1.192459,-1.044483) and (1.261164,-0.971094) .. (1.356392,-0.906312)
  .. controls (1.451621,-0.841531) and (1.575739,-0.784459) .. (1.718416,-0.711791)
  .. controls (1.861093,-0.639123) and (2.020776,-0.548110) .. (2.148065,-0.427277)
  .. controls (2.275353,-0.306443) and (2.365932,-0.155784) .. (2.380211,0.000000)
  -- cycle
}

% Repeating diagonal hatch.
\pgfdeclarepatternformonly{sfhatchfast}
  {\pgfpoint{-.3pt}{-.3pt}}
  {\pgfpoint{3.5pt}{8.3pt}}
  {\pgfpoint{3.2pt}{8pt}}{
    \pgfsetlinewidth{.4pt}
    \pgfpathmoveto{\pgfpoint{3.2pt}{0pt}}
    \pgfpathlineto{\pgfpoint{0pt}{8pt}}
    \pgfusepath{stroke}
  }

\definecolor{boundary}{HTML}{353535}
\definecolor{levelred}{HTML}{EF4148}
\definecolor{faintred}{HTML}{E9A2A5}
\definecolor{hatchgray}{HTML}{969696}

\makebox[\linewidth][c]{%
\resizebox{1.2\linewidth}{!}{%
\begin{tikzpicture}[
  x=1cm,y=1cm,
  line cap=round,line join=round,
  >=stealth,
  every node/.style={font=\large,text=boundary},
  outline/.style={draw=boundary,line width=.85pt},
  redcurve/.style={draw=levelred,line width=.95pt},
  continuation/.style={
    draw=faintred,line width=.75pt,
    dash pattern=on 2pt off 1.3pt
  },
  contact/.style={
    fill=white,draw=boundary,line width=.4pt
  },
  hatch/.style={
    pattern=sfhatchfast,pattern color=hatchgray
  }
]

\path[use as bounding box]
  (-2.8,-3.08) rectangle (9.9,3.02);

% Left panel: shaded region U_r.
\begin{scope}
  \clip (0,0) circle (1.85);
  \path[hatch,even odd rule]
    (0,0) circle (1.85) \SFSnowflake;
\end{scope}

% Complete curve C_r and its solid part F_r.
\draw[continuation] \SFSnowflake;
\begin{scope}
  \clip (0,0) circle (1.85);
  \draw[redcurve] \SFSnowflake;
\end{scope}

% Unit circle and the arcs E_r.
\draw[boundary!55,line width=.65pt]
  (0,0) circle (1.85);

\foreach \a/\b in {
  3.95149/86.04851,
  93.95149/176.04851,
  183.95149/266.04851,
  273.95149/356.04851
}{
  \draw[boundary,line width=1.3pt]
    (\a:1.85)
    arc[start angle=\a,end angle=\b,radius=1.85];
  \filldraw[contact] (\a:1.85) circle (.022);
  \filldraw[contact] (\b:1.85) circle (.022);
}

% Left-panel labels.
\fill[boundary] (0,0) circle (.028);
\node[below] at (.04,-.10) {$0$};
\node[right] at (1.52,-1.54) {$\mathbb D$};

\node at (0,.65) {$F_r$};
\draw[boundary,thin,-]
  (.29,.83) -- (1.192,1.192);

\node[right] at (-2.66,1.12) {$U_r$};
\draw[boundary,thin,-]
  (-2.07,1.12) -- (-1.1,1.33);

\node at (1.17,2.32) {$E_r$};
\draw[boundary,thin,-]
  (1.15,2.12) -- (1.061,1.516);

% Mapping arrow.
\draw[boundary,line width=.85pt,->]
  (2.25,0) -- (3.65,0);
\node at (2.96,.41) {$f$};

% Right panel.
\begin{scope}[shift={(6.40,0)}]

  % Shaded image f_alpha(U_r).
  \begin{scope}
    \clip \SFOmega;
    \path[hatch,even odd rule]
      \SFOmega (0,0) circle (1.86);
  \end{scope}

  % Circle |w|=r.
  \draw[continuation] (0,0) circle (1.86);
  \begin{scope}
    \clip \SFOmega;
    \draw[redcurve] (0,0) circle (1.86);
  \end{scope}

  % Boundary and contact points.
  \draw[outline] \SFOmega;
  \foreach \angle in {22.5,67.5,...,337.5}{
    \filldraw[contact] (\angle:1.86) circle (.022);
  }

  \fill[boundary] (0,0) circle (.028);
  \node[below] at (.03,-.10) {$0$};

  \node[font=\normalsize] at (0,.51)
    {$\Omega=f(\mathbb D)$};

  \node[right] at (0.56,-1.97)
    {$\partial\Omega$};

\end{scope}

% Right-panel labels.
\node[font=\normalsize] at (6.32,2.76)
  {$f(F_r)=\Omega\cap \{z:|z|=r\}$};
\draw[boundary,thin,-]
  (6.85,2.56) -- (6.4,1.86);

\node[font=\normalsize,right] at (8.82,.93)
  {$f(U_r)$};
\draw[boundary,thin,-]
  (8.93,.70) -- (8.49,.12);

% Definitions below the figure.
\node[font=\normalsize] at (3.55,-2.76)
  {$F_r:=\{z\in\D:|f(z)|=r\},\quad
    E_r:=\{\omega\in\partial\mathbb D\setminus N_f:
           |f(\omega)|>r\}.$};

\end{tikzpicture}%
}%
}
\par
\endgroup

\caption{The red curves represent $F_r$
and $f(F_r)$, while the shaded portions are the sets $U_r$ and
$f(U_r)$. The interior snowflake curve
$F_r\subset\mathbb D$ is nonrectifiable.}
\label{fig:snowflake-level-sets}

\end{figure}

It is natural to ask whether the geometric characterizations in
\cite[Theorems~1.1 and~1.2]{KK} extend to quasiconformal mappings. This is not trivial since \cite{KK} heavily relies on conformal invariants which are not quasiconformally invariant. We answer this question as follows.

\begin{thm}\label{mainpositive}
Let $p>0$ and $\alpha>-1$, and let $f:\D\to \Omega$ be a
quasiconformal homeomorphism. Then the following conditions
are equivalent:
\begin{enumerate}
    \item[(1)] $\displaystyle f\in A_{\alpha}^{p}(\mathbb{D})$;
\vspace{2 mm}
\item[(2)] $\displaystyle \int_0^{\infty}r^{p-1} |U_r|^{\frac{\alpha+2}{2}}\,dr<\infty$;
\vspace{2 mm}
    \item[(3)]
        $\displaystyle\int_{0}^{\infty}
        r^{p-1}\mathcal{H}^{1}(E_r)^{\alpha+2}\,dr
        <\infty$;
\vspace{2 mm}
    \item[(4)] $\displaystyle a_f\in A_{\alpha+p}^{p}(\mathbb{D})$.
\end{enumerate}
\vspace{2 mm}
When $-1<\alpha\leq 0$, the above conditions are further
equivalent to:
\vspace{2 mm}
\begin{enumerate}
    \item[(5)]
        $\displaystyle\int_{0}^{\infty}
        r^{p-1}\mathcal{H}^{1}(F_r)^{\alpha+2}\,dr
        <\infty.$
\end{enumerate}
Moreoever, $(5)$ implies $(3)$ for every $\alpha>-1$ and $p>0$.
\end{thm}

In the setting of conformal mappings, $a_f(z)=|f'(z)|$, and the equivalence of (1) and (4) follows from the classical results on holomorphic functions, see e.g. \cite[page 7]{pelaez2021bergman}. For a $K-$quasiconformal mapping $f$, combining \cite[Lemmas 2.3 and 2.5]{AK} with \cite[Corollary~2.2]{Astala1998} yields
$$a_f\in A_{\alpha+p}^{p}(\mathbb D)
\quad\Longleftrightarrow\quad
|Df|\in A_{\alpha+p}^{p}(\mathbb D)
\quad \text{for}\,\,\alpha>-1,\,\,0<p<\frac{2K}{K-1}.$$
Together with the holomorphic case, this suggests that one should have 
$$f\in A_{\alpha}^{p}(\mathbb{D}) \iff |Df|\in A_{\alpha+p}^{p}(\mathbb{D})\quad \text{for}\,\,\alpha>-1,\,\,0<p<\frac{2K}{K-1},$$
when $f$ is $K-$quasiregular. This turns out to be the case; see \cite{KLPqr}.

Note that Theorem~\ref{mainpositive} uses the Hausdorff measure of $F_r$, whereas \cite[Theorem 1.1]{KK} uses its harmonic measure. The Hausdorff measure provides an additional characterization, since the harmonic measure characterization also extends to quasiconformal settings. More precisely, for $p>0$ and $\alpha>-1$,
$$
f\in A_\alpha^p(\D)
\quad\Longleftrightarrow\quad
\int_0^\infty
r^{p-1}\omega_\D(0,F_r)^{\alpha+2}\,dr<\infty;
$$ 
We leave the verification to the interested reader. 

The equivalence of (1) and (5) is more delicate. For conformal mappings, it holds for every $p>0$ and $\alpha>-1$.
The key additional estimate is
\begin{equation}\label{eq:Hausdorff-to-harmonic}
    \HH^1(F_r)\le C\omega_\D(0,F_r),
    \quad \text{for}\,\,r>|f(0)|,
\end{equation}
where $C$ is an absolute constant. To establish \eqref{eq:Hausdorff-to-harmonic}, first apply Hayman-Wu to get local-length control. Then local-length control and a Green function argument gives the desired harmonic measure estimate. However, the Hayman--Wu length bound fails for general quasiconformal mappings \cite[Theorem 1.1]{AstalaFernandezRohde1993}, even if the bound is allowed to
depend on the mapping. Indeed, our snowflake construction below gives a single bounded quasiconformal
mapping for which $\HH^1(F_r)$ is unbounded as $r$ varies. 

The same construction shows that the range $-1<\alpha\le0$ is optimal in the quasiconformal setting. Specifically, the following theorem shows that, for $\alpha>0$,
condition~(1) need not imply condition~(5), even for bounded
quasiconformal mappings.

\begin{thm}\label{thm:weighted-counterexample}
For every $\alpha>0$ there exist a $K_\alpha\geq1$ and a bounded
$K_\alpha-$quasiconformal homeomorphism $f_\alpha:\D\to\Omega_\alpha$ with $f_\alpha(0)=0$ such that for every $p>0$, $f_\alpha\in A_\alpha^p(\mathbb{D})$ but
$$
\int_0^\infty
r^{p-1}\HH^1(F_r)^{\alpha+2}\,dr=\infty.
$$
\end{thm}

It would be interesting to know if, for fixed $K$, there exists $\epsilon=\epsilon(K)$ such that for $0<\alpha< \epsilon(K)$, $(1)\implies (5)$ for every $K$-quasiconformal map.

This paper is organized as follows. In Section \ref{sec:Preliminaries}, we introduce the necessary notation, recall some basic properties of quasiconformal mappings, and establish several auxiliary lemmas that play a key role in the proof of the main result and are also of independent interest. Sections \ref{sec:weighted} and \ref{Example} are devoted to the proofs of Theorems \ref{mainpositive} and \ref{thm:weighted-counterexample}, respectively.
	
\section{Preliminaries}\label{sec:Preliminaries}

    \subsection{Notation}
We denote the unit disk in $\mathbb C$ by $\D$ and $\HH^1$ stands
for one-dimensional Hausdorff measure. 
For a measurable set $V\subset\mathbb C$, we write $|V|$
for its Lebesgue area. If $V\subset\D$ and $\alpha>-1$,
we define its weighted area by 
$$|V|_\alpha
:=
\int_V (1-|z|)^\alpha\,dA(z),$$
where $dA$ denotes Lebesgue area.
In particular, $|V|_0=|V|$.

At each point where a mapping $g$ is differentiable,
we denote the operator norm of its differential by
$$|Dg(z)|
:=
\sup_{v\neq 0}\frac{|Dg(z)v|}{|v|}
=
\sup_{|v|=1}|Dg(z)v|.$$
    
For a set $A\subset\D\setminus\{0\}$, we define its radial
projection onto $\partial\D$ by
\begin{equation}\label{eq:radial-projection}
\pi(A)
:=
\left\{
\frac{z}{|z|}:z\in A
\right\}.
\end{equation}

For $0<a<b<\infty$, we set
$$A(0;a,b)
:=
\{z\in\mathbb C:a<|z|<b\}$$
to be the annulus centered at the origin with inner radius $a$
and outer radius $b$.

    \subsection{Quasiconformal mappings}

    \begin{defn}
        A homeomorphism $f:\D\to \Omega\subset\mathbb{C}$ is $K-$quasiconformal if $f$ belongs to the Sobolev class $W_{loc}^{1,2}(\D;\mathbb{C})$ and 
        $$|Df(x)|^2\le KJ_f(x)\quad\text{for almost every} \quad x\in\D.$$
    \end{defn}
    
	We recall the following result, which is the planar version of the quasiconformal Koebe theorem due to Astala and Gehring \cite{AG}.
	\begin{lem}\label{lem:koebe}
		Let $f:\mathbb{D}\to \Omega$ be $K-$quasiconformal. Then there is a constant $C=C(K)$, so that for every $z\in \mathbb{D}$,
		$$
		\dist(f(z),\partial\Omega)/C\le a_f(z)(1-|z|)
		\leq C\dist(f(z),\partial\Omega).
		$$
	\end{lem}

	\begin{lem}\label{lem:BKL}
    Let $p>0$ and $\alpha>-1$, and let $f:\D \to \Omega$ be a $K$-quasiconformal homeomorphism. Then
    $$\int_0^1M(r)^p(1-r)^{\alpha+1}\,dr<\infty\quad\Longleftrightarrow\quad\int_\D
		a_f(z)^p(1-|z|)^{p+\alpha}\,dA(z)<\infty.$$
	\end{lem}
	
	This is a part of \cite[Theorem~1.1]{BKL} with $n=2$.
    We next recall the planar case of a result by Astala and the first author \cite[Lemma~4.2]{AK}. For $z\in\D$, set
		$$
		S(z):=\partial\D\cap B\bigl(z,3(1-|z|)\bigr),
		$$
		and let $f(\omega)$ denote the radial limit of $f$ at $\omega$, which
		exists for $\HH^1$-almost every $\omega\in \partial\D$.
	
		\begin{lem}\label{lem:Astala-Koskela-annulus-estimate}
        Let $f:\D \to \Omega$ be a $K$-quasiconformal homeomorphism with $f(z)\neq0$ for all $z\in\D$.  Then, for every $z\in\D$ and $M>1$,
		$$
		\HH^1\left(
		\left\{\omega\in S(z):
		|f(\omega)|<\frac{|f(z)|}{M}
		\right\}\right)
		\leq
		\frac{C(K)}{\log M}\HH^1(S(z)).
		$$
	\end{lem}

	\subsection{Preliminary results}\label{sec:Supporting-lemma}
In this section we prove some lemmas which will be useful towards our main theorems. 
\begin{lem}\label{cavalieri}
Let $p>0$, $\alpha>-1$, and let
$f:\D\to\mathbb C$ be measurable. Then
    \begin{equation}\label{eq:layer-cake}
\int_{\D}|f(z)|^p (1-|z|)^{\alpha}\,dA(z)=p\int_0^\infty r^{p-1}|U_r|_{\alpha}\,dr.
\end{equation}
\end{lem}

\begin{pf}
Apply Tonelli's theorem to the pointwise identity
$$
|f(z)|^p
=
p\int_0^\infty
r^{p-1}\mathbf1_{\{|f(z)|>r\}}\,dr.
$$
\end{pf}

\begin{comment}
For every $z\in\D$,
$$
|f(z)|^p
=
p\int_0^\infty
r^{p-1}\mathbf 1_{\{|f(z)|>t\}}\,dr.
$$
Multiplying by $(1-|z|)^\alpha$, integrating over $\D$, and applying Tonelli's theorem to the nonnegative integrand gives
$$
\begin{aligned}
	\int_{\D}|f(z)|^p(1-|z|)^\alpha\,dA(z)
	&=
	p\int_{\D}\int_0^\infty
	r^{p-1}\mathbf 1_{\{|f(z)|>t\}}
	(1-|z|)^\alpha\,dr\,dA(z)\\
	&=
	p\int_0^\infty r^{p-1}
	\left(
	\int_{U_r}(1-|z|)^\alpha\,dA(z)
	\right)dr\\
	&=
	p\int_0^\infty r^{p-1}|U_r|_\alpha\,dr.
\end{aligned}
$$   
\end{comment}

Recall that we denote
$$F_r:=\{z\in\D:|f(z)|=r\}=f^{-1}(S(0,r)\cap \Omega)
		\quad\text{for}\quad r>0.$$

\begin{lem}\label{lem:level-set-slicing}
Let $f:\D\to \Omega$ be a $K-$quasiconformal homeomorphism, denote $g:=f^{-1}:\Omega\to\D$, and, for $w\neq0$, let
$\tau(w):=iw/|w|$. We choose a Borel measurable representative of $Dg$ and define
\begin{equation}\label{defLr}
    L(r):=
	\int_{\Omega\cap S(0,r)}
	|Dg(w)\tau(w)|\,d\HH^1(w),
	\quad \text{for}\quad r>0.
\end{equation}
	Then $L:(0,\infty)\to[0,\infty]$ is measurable. Moreover, for almost
	every $r>0$, $F_r$ is countably $1$-rectifiable and
\begin{equation}\label{eq:Hausdorff-F_r-Lr}
    \HH^1(F_r)
	=
	L(r)
	\leq
	\int_{\Omega\cap S(0,r)}
	|Dg(w)|\,d\HH^1(w)<\infty.
\end{equation}
	In particular, $L$ is a measurable representative of the function
	$r\mapsto\HH^1(F_r)$.
\end{lem}

\begin{proof}
Since $g:\Omega\to\mathbb D$ is also $K$-quasiconformal, the change of variables formula gives
$$\int_\Omega |Dg(w)|^2\,dA(w)
    \le K\int_\Omega J_g(w)\,dA(w)
    =K\pi.$$
Consequently, polar coordinates and the Cauchy--Schwarz
inequality imply that, for almost every $r>0$,
$$
\int_{\Omega\cap S(0,r)}|Dg(w)|\,d\HH^1(w)
\le
(2\pi r)^{1/2}
\left(
\int_{\Omega\cap S(0,r)}
|Dg(w)|^2\,d\HH^1(w)
\right)^{1/2}
<\infty.
$$

By the Sobolev chain rule and slicing in local polar
coordinates \cite[Section~4.9.2]{EG}, together with the
continuity of $g$, for almost every such $r$ the restriction
of $g$ to each component of $\Omega\cap S(0,r)$ is locally
absolutely continuous, with tangential derivative
$Dg(w)\tau(w)$. 

Applying the length formula on compact subarcs, using
injectivity, and exhausting the countably many components,
we obtain that $F_r$ is countably $1$-rectifiable and
$$
\HH^1(F_r)
=
\int_{\Omega\cap S(0,r)}
|Dg(w)\tau(w)|\,d\HH^1(w)
=L(r).
$$
The asserted inequality follows from $|\tau(w)|=1$.

Finally, extending $Dg$ by zero outside $\Omega$, we have
$$
L(r)
=
r\int_0^{2\pi}
|Dg(re^{i\theta})(ie^{i\theta})|\,d\theta,
\quad r>0.
$$
Since the integrand is nonnegative and Borel measurable in
$(r,\theta)$, by Tonelli's theorem $L$ is measurable.

\end{proof}

Recall that for a measurable set $V\subset\D$, we denote the weighted area of $V$ by
$$|V|_\alpha:=\int_V(1-|z|)^\alpha\,dA(z),
$$
and let
$$S(z):=\partial\mathbb D\cap B\bigl(z,3(1-|z|)\bigr).$$
	
\begin{lem}\label{lem:tents}
Let $G\subset\partial\mathbb D$ and $V\subset\mathbb D$
be measurable sets satisfying 
\begin{equation}\label{densityassumption}
    \mathcal H^1\bigl(G\cap S(z)\bigr)
    \ge \frac12\mathcal H^1\bigl(S(z)\bigr)
    \quad\text{for every }z\in V.
\end{equation}
	Then, for every $\alpha>-1$, we have
    \begin{equation}\label{eq:weighted-tent}
    |V|_\alpha
    :=\int_V(1-|z|)^\alpha\,dA(z)
    \le \frac{2}{\alpha+1}\,
        \mathcal H^1(G)^{\alpha+2}.
\end{equation}
	\end{lem}

    \begin{proof}
For $z=\rho e^{i\theta}\in\D$, the arc
$\{e^{i(\theta+t)}:|t|<1-\rho\}$ lies in $S(z)$.
Hence
$$
\HH^1(S(z))\ge 2(1-|z|),
$$
and \eqref{densityassumption} gives
\begin{equation}\label{inclusionofV}
V\subset
\{z\in\D:1-|z|\le\HH^1(G)\}.
\end{equation}
In particular, if $\HH^1(G)=0$, then $V=\emptyset$. By rotational invariance and Tonelli's theorem, we have
$$
\int_0^{2\pi}
\frac{\HH^1(G\cap S(\rho e^{i\theta}))}
     {\HH^1(S(\rho e^{i\theta}))}\,d\theta
=
\HH^1(G),
\quad 0<\rho<1.
$$
Using \eqref{densityassumption},
\eqref{inclusionofV} and polar coordinates, we obtain
\begin{align*}
|V|_\alpha
&\le
2\int_{\max\{0,1-\HH^1(G)\}}^1
(1-\rho)^\alpha\rho
\int_0^{2\pi}
\frac{\HH^1(G\cap S(\rho e^{i\theta}))}
     {\HH^1(S(\rho e^{i\theta}))}\,d\theta\,d\rho\\
&=
2\HH^1(G)
\int_{\max\{0,1-\HH^1(G)\}}^1
(1-\rho)^\alpha\rho\,d\rho\\
&\le
2\HH^1(G)\int_0^{\HH^1(G)}s^\alpha\,ds\\
&=
\frac{2}{\alpha+1}\HH^1(G)^{\alpha+2}.
\end{align*}
\end{proof}

Recall that 
$$ E_r:=\{\omega\in \partial\D\setminus N_f:|f(\omega)|>r\}$$
for the given quasiconformal homeomorphism $f:\D\to \Omega$.

\begin{lem}\label{lem:shifted-density}
Let $f:\mathbb{D}\to \Omega$ be a $K$-quasiconformal homeomorphism, and set $d_0:=\dist(0,\partial\Omega)$. Then there exists a constant $\Lambda=\Lambda(K)>1$, such that
for every positive $r\ge d_0$ and every $z\in\D$ with
$|f(z)|\ge\Lambda r$, we have
    $$\HH^1(E_r\cap S(z))
    \ge \frac12\HH^1(S(z)).$$
	\end{lem}
	
	\begin{proof}
    We choose
	$y_0\in\partial\Omega$ with $|y_0|=d_0$ and let $h:=f-y_0$. Then $h:\D \to \Omega-y_0$ is $K$-quasiconformal and $h(z)\neq0$ for every $z\in \D$. Let $C(K)$ be the constant in
	Lemma~\ref{lem:Astala-Koskela-annulus-estimate}. Choose a constant $M_0=M_0(K)>1$ such that ${C(K)}/{\log M_0}\leq 1/2$,
		and set $\Lambda=\Lambda(K):=2M_0+2>1$.
    If positive $r\geq d_0$ and $|f(z)|\geq\Lambda r$, then
		$$
			|h(z)|\geq |f(z)|-|y_0|\geq \Lambda r-d_0\geq(2M_0+1)r.$$
	
 For every $\omega\in S(z)\setminus(E_r\cup N_f)$, we therefore have
        $$|h(\omega)|
    =|f(\omega)-y_0|
    \le r+d_0 \le2r
    <\frac{|h(z)|}{M_0}.$$ 
  Applying Lemma \ref{lem:Astala-Koskela-annulus-estimate} to $h$ and using $\HH^1(N_f)=0$ yields that
\begin{align*}
    \HH^1(S(z)\setminus E_r)
    &\le
    \HH^1\left(
    \left\{\omega\in S(z)\setminus N_f:
    |h(\omega)|<\frac{|h(z)|}{M_0}\right\}
    \right)\\
    &\le \frac{C_K}{\log M_0}\HH^1(S(z))\\
    &\le \frac12\HH^1(S(z)).
\end{align*}
Therefore, we conclude that
$$\HH^1(E_r\cap S(z))=\HH^1(S(z))- \HH^1(S(z)\setminus E_r)\ge \frac{1}{2}\HH^1(S(z)),$$
which proves the lemma.
	\end{proof}

Recall that for each $r\ge0$, we defined $$U_r:=\{z\in \mathbb D: |f(z)|>r\}=f^{-1}(\Omega\cap \{|z|>r\}),$$
\begin{lem}\label{Urestimate}
Let $\alpha>-1$, and let $f:\D \to \Omega$ be a $K$-quasiconformal homeomorphism. Then there exist constants $\Lambda=\Lambda(K)>1$ such that
\begin{equation}\label{eq:weighted-shifted-superlevel-upper-bound}
			|U_{\Lambda r}|_\alpha
			\leq
			\frac{2}{\alpha+1}\,\HH^1(E_r)^{\alpha+2},\quad \text{for every } r>0,
		\end{equation}
and
\begin{equation}\label{eq:weighted-shifted-superlevel-lower-bound}
			|U_r|_\alpha
			\geq
	C(K,\alpha)\,\HH^1(E_{2r})^{\alpha+2},\quad \text{for every } r>R_0,
		\end{equation}        
        where $$R_0:=\max_{|z|=\frac{1}{2}}|f(z)|,\quad C(K,\alpha)=\frac{1}{2(\alpha+1)}
\left(\frac{\log 2}{4\pi K}\right)^{\alpha+1}.$$
\end{lem}

\begin{proof}
Recall that $d_0:=\dist(0,\partial\Omega)$. 
By Lemma~\ref{lem:shifted-density}, for $r\ge d_0$, we have
$$
\HH^1(E_r\cap S(z))
\geq
\frac12\HH^1(S(z)),\quad z\in U_{\Lambda r} .
$$
Combining with Lemma \ref{lem:tents}, it follows that
		$$
		|U_{\Lambda r}|_\alpha
		\leq
		\frac{2}{\alpha+1}\,\HH^1(E_r)^{\alpha+2}.
		$$

        On the other hand, if $0<r<d_0$, then $\HH^1(E_r)=2\pi$. Thus,
		$$|U_{\Lambda r}|_\alpha\le
		|\D|_\alpha
		=
		\frac{2\pi}{(\alpha+1)(\alpha+2)}\le \frac{2}{\alpha+1}\,\HH^1(E_r)^{\alpha+2}.
		$$
		The last inequality follows from $\alpha>-1$. This gives \eqref{eq:weighted-shifted-superlevel-upper-bound} for every $r>0$. 

    To prove the lower bound \eqref{eq:weighted-shifted-superlevel-lower-bound}, fix $r>R_0$. We may assume that $\mathcal H^1(E_{2r})>0$. For $\omega\in E_{2r}$, define 
$$
a(\omega)
:=\max\{t\in[1/2,1):|f(t\omega)|\leq r\},\quad\delta(\omega):=1-a(\omega).
$$
Then $a(\omega)>1/2$, $|f(a(\omega)\omega)|=r$ and the radial segment $\{t\omega:a(\omega)<t<1\}$ lies in $U_r$.

Let $\Gamma$ consist of the radial segments from $a(\omega)\omega$ to the first point $b(\omega)$ where $|f(b(\omega)\omega)|=2r$. Since $f(\gamma_\omega)\subset A(0;r,2r)$ for every $\gamma_\omega\in \Gamma$, by Cauchy--Schwarz and the quasiconformality of $f$, it follows that

\begin{equation}\label{eq:U_r-modulus}
    \int_{E_{2r}}\frac{d\mathcal H^1(\omega)}{\delta(\omega)} \le 2\operatorname{Mod}(\Gamma) \le  2K\operatorname{Mod}(f\Gamma)\le \frac{4\pi K}{\log 2}.
\end{equation}

Integrating over the radial segments and applying change of variables, gives
\begin{equation}\label{eq:U_r-lower-bound}
    |U_r|_{\alpha}\ge \frac12\int_{E_{2r}}\int_{a(\omega)}^1
(1-t)^\alpha\,dt\,d\mathcal H^1(\omega)\ge\frac{1}{2(\alpha+1)}
\int_{E_{2r}}\delta(\omega)^{\alpha+1}\,d\mathcal{H}^1(\omega).
\end{equation}

Finally, since $\alpha>-1$, by H\"older's inequality, we have
\begin{equation}\label{eq:U_r-Holder}
\mathcal{H}^1(E_{2r})^{\alpha+2}
\leq
\left(
\int_{E_{2r}}\delta^{\alpha+1}\,d\mathcal{H}^1
\right)
\left(
\int_{E_{2r}}\delta^{-1}\,d\mathcal{H}^1
\right)^{\alpha+1}.
\end{equation}
Using \eqref{eq:U_r-modulus} and \eqref{eq:U_r-lower-bound} in \eqref{eq:U_r-Holder} gives \eqref{eq:weighted-shifted-superlevel-lower-bound} with $C(K,\alpha)=\frac{1}{2(\alpha+1)}
\left(\frac{\log 2}{4\pi K}\right)^{\alpha+1}$.
\end{proof}

\section{Proof of Theorem \ref{mainpositive}}\label{sec:weighted}

This section is devoted to the proof of Theorem~\ref{mainpositive}. Our strategy is as follows. Throughout, we assume that $p>0$. For $\alpha>-1$, we establish $(5)\Rightarrow(3)$,
$(1)\Leftrightarrow(3)$, and $(2)\Leftrightarrow(3)$
in subsections~\ref{6implies3}--\ref{sect:areacharact}, respectively.
In subsection~\ref{equ145}, we prove the equivalence of conditions
$(1)$ and $(4)$. In view of $(5)$ implies $(3)$, these results show that $(5)$ implies all the conditions for every $\alpha>-1$.
Finally, in subsection~\ref{1implies6}, we prove $(1)\Rightarrow(5)$
for $-1<\alpha\le0$, completing the proof of Theorem~\ref{mainpositive}.

\subsection{(5) implies (3)}\label{6implies3}
We first show that $(5)\implies (3)$ for $\alpha>-1$.

\begin{thm}\label{thm:(5)implies(3)}
Let $p>0$ and $\alpha>-1$, and let $f:\mathbb{D}\to \Omega$ be a $K$-quasiconformal homeomorphism. Then 
$$
\int_{0}^\infty r^{p-1}\HH^1(F_r)^{\alpha+2}\,dr<\infty\quad\Longrightarrow\quad
\int_0^\infty r^{p-1}\HH^1(E_r)^{\alpha+2}\,dr<\infty.
$$
\end{thm}
	
\begin{pf}
Recall that we denote $0<R_0:=\max_{|z|=\frac{1}{2}}|f(z)|<\infty$. Since $f$ is continuous and open, $|f(z)|\le R_0$
for $|z|\le1/2$. Thus, for $r>R_0$,
$$
F_r\subset\{z\in\D:|z|>1/2\}.
$$
Moreover, for every $\omega\in E_r$, we have
$$
\left|f\left({\omega}/{2}\right)\right|\le R_0<r<|f(\omega)|.
$$
By the definition of radial limits and the intermediate
value theorem, the radial segment $\{s\omega:1/2<s<1\}$
meets $F_r$. Hence $E_r\subset\pi(F_r)$.

Since the radial projection $\pi(z)=z/|z|$ is $4$-Lipschitz on $\{|z|\ge1/2\}\subset \D\setminus\{0\}$, it follows that
$$
\HH^1(E_r)\le\HH^1(\pi(F_r))\le4\HH^1(F_r),
\quad \text{for every}\,\,\, r>R_0.
$$
By the assumption (5), we know that
$$\int_{R_0}^\infty r^{p-1}\HH^1(F_r)^{\alpha+2}\,dr
<\infty$$
Combining this with $\HH^1(E_r)\le2\pi$, we conclude that
$$
\int_0^\infty r^{p-1}\HH^1(E_r)^{\alpha+2}\,dr
\le
\frac{(2\pi)^{\alpha+2}}{p}R_0^p
+
4^{\alpha+2}
\int_{R_0}^\infty r^{p-1}\HH^1(F_r)^{\alpha+2}\,dr
<\infty,
$$
which proves the theorem.
\end{pf}

\subsection{Equivalence of (1) and (3)}\label{thm:(3)iff(1)}

\begin{thm}\label{1iff3}
Let $p>0$ and $\alpha>-1$, and let
$f:\mathbb{D}\to\Omega$ be a $K$-quasiconformal homeomorphism.
Then
$$
f\in A_\alpha^p(\mathbb{D})
\quad\Longleftrightarrow\quad
\int_0^\infty r^{p-1}\HH^1(E_r)^{\alpha+2}\,dr<\infty.
$$
\end{thm}

\begin{pf}
The claim follows from Lemma \ref{cavalieri} together with Lemma \ref{Urestimate}. Indeed, by integrating \eqref{eq:weighted-shifted-superlevel-upper-bound}
and \eqref{eq:weighted-shifted-superlevel-lower-bound}
against $r^{p-1}\,dr$, a change of variables gives
$$\int_0^\infty r^{p-1}|U_r|_\alpha\,dr
\le
C_1\int_0^\infty r^{p-1}\HH^1(E_r)^{\alpha+2}\,dr,$$
and 
$$\int_{2R_0}^\infty r^{p-1}\HH^1(E_r)^{\alpha+2}\,dr
\le
C_2\int_{R_0}^\infty r^{p-1}|U_r|_\alpha\,dr,$$
where $C_1,C_2\ge 1$ depend only on $K,p,\alpha$, and $R_0:=\max_{|z|=\frac{1}{2}}|f(z)|$.
Since $\HH^1(E_r)\le2\pi$ and $p>0$, we have
\begin{equation}\label{int02R0}
    \int_{0}^{2R_0} r^{p-1}\HH^1(E_r)^{\alpha+2}\,dr\le \frac{(2\pi)^{\alpha+2}\cdot (2R_0)^p}{p}<\infty.
\end{equation}

Combining the above estimates, we conclude that
$$
\int_0^\infty r^{p-1}|U_r|_\alpha\,dr<\infty
\quad\Longleftrightarrow\quad
\int_0^\infty r^{p-1}\HH^1(E_r)^{\alpha+2}\,dr<\infty.
$$
The claim now follows from Lemma \ref{cavalieri}.
\end{pf}

\subsection{Equivalence of (2) and (3)}\label{sect:areacharact}

\begin{thm}\label{2dimcharacterization}
Let $p>0$ and $\alpha>-1$, and let
$f:\mathbb D\to\Omega$ be a $K$-quasiconformal
homeomorphism. Then,
$$ \int_0^\infty r^{p-1}|U_r|^{\frac{\alpha+2}{2}}\,dr<\infty
\quad\Longleftrightarrow\quad
\int_0^\infty r^{p-1}\HH^1(E_r)^{\alpha+2}\,dr<\infty.$$
\end{thm}
                             
\begin{proof}

Applying Lemma \ref{Urestimate} with $\alpha=0$, we obtain
$$|U_{\Lambda r}|\le 2\HH^1(E_r)^2 \quad \text{for every } \,\,r>0,$$
and
$$|U_{r}|\ge C(K)\HH^1(E_{2r})^2
	\quad \text{for every }\,\, r>R_0,$$
where $\Lambda=\Lambda(K)>1$, $C(K)>0$ and $R_0:=\max_{|z|=\frac{1}{2}}|f(z)|$. Raising both inequalities to the power $(\alpha+2)/2$, integrating against $r^{p-1}\,dr$ and using a change of variables, we obtain
$$\int_0^\infty r^{p-1}|U_r|^{\frac{\alpha+2}{2}}\,dr
	\le
	C_1\int_0^\infty r^{p-1}\HH^1(E_r)^{\alpha+2}\,dr,$$
    and
    $$\int_{2R_0}^\infty r^{p-1}\HH^1(E_r)^{\alpha+2}\,dr
	\le
	C_2\int_{R_0}^\infty r^{p-1}|U_r|^{\frac{\alpha+2}{2}}\,dr,$$
where $C_1,C_2\ge 1$ depend only on $K,p,\alpha$. 
Combining these estimates with \eqref{int02R0} proves the claim.
\end{proof}

\subsection{Equivalence of (1) and (4)}\label{equ145}

We now establish an energy characterization of quasiconformal mappings in weighted $L^p$ spaces, which is also of independent interest.

\begin{thm}\label{thm:f-iff-a_f}
Let $p>0$ and $\alpha>-1$, and let $f:\mathbb{D}\to \Omega$ be a $K$-quasiconformal homeomorphism. Then
$$  \int_\D
|f(z)|^{p}(1-|z|)^{\alpha}\,dA(z)<\infty \iff \int_\D a_f(z)^{p}(1-|z|)^{p+\alpha}\,dA(z)<\infty.$$
\end{thm}

\begin{pf}
     Suppose $\int_\D
|f(z)|^{p}(1-|z|)^{\alpha}\,dA(z)<\infty$. Choose $y_0\in\partial\Omega$ with $|y_0|=d_0:=\dist(0,\partial\Omega)$. By
		Lemma~\ref{lem:koebe}, we have
		$$
a_f(z)(1-|z|)
\le C(K)\dist(f(z),\partial\Omega)
\le C(K)(|f(z)|+d_0).
		$$
		Consequently, since $\alpha>-1$, 
	$$\int_{\mathbb D}a_f(z)^p(1-|z|)^{p+\alpha}\,dA(z)
\le C(K,p)\int_{\mathbb D}
\bigl(|f(z)|^p+d_0^p\bigr)(1-|z|)^\alpha\,dA(z)
<\infty.
$$

Conversely, by the equivalence of (1) and (3) established in Section \ref{thm:(3)iff(1)}, it suffices to prove that $\int_\D a_f(z)^{p}(1-|z|)^{p+\alpha}\,dA(z)<\infty$ implies (3). By Lemma~\ref{lem:BKL}, we have
\begin{equation}\label{eq:maximum-modulus-integral}
   \int_0^1 M(s)^p(1-s)^{\alpha+1}\,ds<\infty.
\end{equation}
Recall that $R_0:=M(1/2)$ and set $R:=\sup_{\D}|f|$, allowing
$R=\infty$. Since $f$ is continuous and open, $M$ is continuous and strictly increasing.
Thus, for each $2R_0<t<2R$, there is a unique
$\rho_t\in(1/2,1)$ such that
\begin{equation}\label{Mrhot}
    2M(\rho_t)=t.
\end{equation}

Let $\Gamma_t$ be the family of radial segments from $\rho_t\omega$ to $\omega$,
with $\omega\in E_t$. By a direct computation in polar coordinates,
$$\Mod(\Gamma_t)=\frac{\HH^1(E_t)}{\log({1}/{\rho_t})}.$$
Since $|f(\rho_t\omega)|\le t/2$ and $|f(\omega)|>t$, continuity implies that each curve in $f(\Gamma_t)$ contains a subcurve joining the two boundary circles of $A(0;t/2,t)$.
Combining these with $\rho_t\in(1/2,1)$ and the quasiconformality of $f$, we deduce that
\begin{equation}\label{eq:m-r}
	\HH^1(E_t)\le K \Mod(f\Gamma_t)\log\left(1/{\rho_t}\right)
	\le
	\frac{4\pi K}{\log2}(1-\rho_t).
\end{equation}

Note that 
$$(1-\rho_t)^{\alpha+2}
=(\alpha+2)\int_{\rho_t}^1(1-s)^{\alpha+1}\,ds.$$
By \eqref{eq:maximum-modulus-integral}, \eqref{Mrhot}, \eqref{eq:m-r} and Tonelli's theorem, we conclude that
\begin{align*}
	\int_{2R_0}^{2R}
	r^{p-1}\HH^1(E_r)^{\alpha+2}\,dr
	&\le
	C(K,\alpha)\int_{2R_0}^{2R}r^{p-1}
	\int_{\rho_r}^1(1-s)^{\alpha+1}\,ds\,dr\\
	&=
	C(K,\alpha)\int_{1/2}^1(1-s)^{\alpha+1}
	\int_{2R_0}^{2M(s)}r^{p-1}\,dr\,ds\\
	&\le
	C(K,p,\alpha)\int_{1/2}^1
	M(s)^p(1-s)^{\alpha+1}\,ds
	<\infty.
\end{align*}
The integral over $(0,2R_0)$ is finite by \eqref{int02R0}. If $R<\infty$,
then $E_r=\emptyset$ for every $r\ge R$. Hence
$$
\int_0^\infty r^{p-1}\HH^1(E_r)^{\alpha+2}\,dr<\infty.
$$
This is (3) and the proof is complete.
\end{pf}

Lemma~\ref{lem:BKL} and Theorem~\ref{thm:f-iff-a_f} give the following corollary.
\begin{cor}\label{cor:Ap-to-integral-mean}
 Let $p>0$ and $\alpha>-1$, and let $f:\mathbb{D}\to \Omega$ be a $K$-quasiconformal homeomorphism. Then
$$ 
f\in A_{\alpha}^{p} \iff \int_0^1 M(s)^p(1-s)^{\alpha+1}\,ds<\infty.$$
\end{cor}

\subsection{(1) implies (5) when $-1<\alpha\leq 0$}\label{1implies6}
\begin{thm}\label{thm:(1)implies(5)}
	Let $p>0$ and $-1<\alpha\leq 0$, and let $f:\mathbb{D}\to \Omega$ be a $K$-quasiconformal homeomorphism. Then
	$$
	f\in A_\alpha^p\quad \implies \quad	\int_0^\infty
	r^{p-1}\HH^1(F_r)^{\alpha+2}\,dr<\infty.
	$$	
\end{thm}

\begin{pf}
	Suppose that $f\in A_\alpha^p(\D)$. Let $g=f^{-1}$,
	and let $L$ be the measurable representative of $r\mapsto\HH^1(F_r)$ defined in Lemma \ref{lem:level-set-slicing}. 
	For $r>0$, define
	$$
	\mathcal A(r):=
	\int_{\Omega\cap S(0,r)}
	|Dg(w)|^2(1-|g(w)|)^\alpha\,d\HH^1(w).
	$$
	By polar coordinates, the quasiconformality of $g$,
	and a change of variables, 
	\begin{align}
		\int_0^\infty r^p\mathcal A(r)\,dr
		&=
		\int_\Omega
		|w|^p|Dg(w)|^2(1-|g(w)|)^\alpha\,dA(w)
		\notag\\
		&\le
		K\int_\D
		|f(z)|^p(1-|z|)^\alpha\,dA(z)
		<\infty.
		\label{eq:weighted-inverse-energy}
	\end{align}
    In particular, $A(r)<\infty$ for almost every $r>0$. By Cauchy–Schwarz inequality and $\alpha\le0$, it follows that
	\begin{equation}\label{eq:weighted-level-CS}
		L(r)^2
		\le
		\mathcal A(r)
		\int_{\Omega\cap S(0,r)}
		(1-|g(w)|)^{-\alpha}\,d\HH^1(w)\le2\pi r\mathcal A(r)
	\end{equation}
	for almost every $r>0$.

	Choose $y\in\partial\Omega$ and fix $R>|y|$.
	For $r>R$, the circle $S(0,r)$ cannot be contained
	in $\Omega$: otherwise, simply connectivity would imply
	$B(0,r)\subset\Omega$, contradicting
	$y\in\partial\Omega$.
	Thus each component $\sigma$ of $\Omega\cap S(0,r)$ is an open circular arc with endpoints in $\partial\Omega$.

    As $w$ approaches either endpoint along $\sigma$, we have $|g(w)|\to 1$; otherwise, compactness and the continuity of $f$ would imply that the limiting endpoint belongs to $\Omega$, a contradiction.
    Consequently, for almost every $r>R$, Lemma~\ref{lem:level-set-slicing} gives
    $$1-|z| \le\mathcal H^1(g(\sigma)) \le L(r) \quad \text{for every}\,\,\,z\in g(\sigma).$$
   Since this estimate holds for every component of $\Omega\cap S(0,r)$, we have
   $$1-|g(w)|\le L(r) \quad \text{for every}\,\,\,w\in \Omega\cap S(0,r).$$
   For $L(r)>0$, combining this with \eqref{eq:weighted-level-CS} and $\alpha\le 0$, it follows that
	$$
	L(r)^2
	\le
	2\pi r\mathcal A(r)L(r)^{-\alpha}.
	$$
	Hence
	\begin{equation*}
		L(r)^{\alpha+2}
		\le2\pi r\mathcal A(r)\quad \text{for almost every}\,\,\, r>R.
	\end{equation*}
 The same inequality is immediate
	when $L(r)=0$.

    For $0<r<R$, by \eqref{eq:weighted-level-CS} and the assumption $1<\alpha+2\le2$, we obtain
	$$
	L(r)^{\alpha+2} \le 1+L(r)^2
	\le1+2\pi r\mathcal A(r) \quad \text{for almost every}\,\,\,0<r<R.
	$$
	 Combining these two estimates with Lemma \ref{lem:level-set-slicing} and
	\eqref{eq:weighted-inverse-energy}, we conclude that
	\begin{align*}
		\int_0^\infty
		r^{p-1}\HH^1(F_r)^{\alpha+2}\,dr
		&=
		\int_0^\infty r^{p-1}L(r)^{\alpha+2}\,dr\\
		&\le
		\frac{R^p}{p}
		+2\pi\int_0^\infty r^p\mathcal A(r)\,dr\\
		&\le
		\frac{R^p}{p}
		+2\pi K\int_\D
		|f(z)|^p(1-|z|)^\alpha\,dA(z)
		<\infty.
	\end{align*}
    This completes the proof.
\end{pf}

\begin{pf}[Proof of Theorem \ref{mainpositive}]
    The claim follows by combining Theorems \ref{thm:(5)implies(3)}- \ref{thm:(1)implies(5)}.
\end{pf}

\section{Proof of Theorem~\ref{thm:weighted-counterexample}}\label{Example}

We prove Theorem~\ref{thm:weighted-counterexample} using a homogeneous
Rohde snowflake as in \cite[Section 3]{HM}. The idea is to construct a quasiconformal map whose
images of concentric circles have lengths growing at least at a
prescribed rate.
After a similarity normalization, the inverse map turns these curves
into the level sets of a bounded quasiconformal mapping.

\subsection{A snowflake parametrization and its quasiconformal extension}

    Fix $\alpha>0$ and let
    \begin{equation}\label{settingofcounter}
        q:=\alpha+2,\quad
    \lambda:=\frac1q,\quad
    \beta:=1-\lambda,\quad
    \rho:=4^{-\beta}.
    \end{equation}
    Then $0<\lambda<1/2$, $1/2<\beta<1$, and $1/4<\rho<1/2$.

    We start with the unit square and apply the standard outward
Rohde snowflake replacement to every edge at every stage, with parameter
$\rho$; see \cite[Section 4.3]{HM}. Each edge is replaced by four edges of
length $\rho$ times the parent length, and all existing vertices are
preserved. Let $\Gamma_\rho$ denote the limiting Jordan curve.
After $n\ge0$ replacements, the approximating polygon has $4^{n+1}$
edges, each of length $\rho^n$.

Let $\mathcal J_m$ be the standard subdivision of $\partial\mathbb D$
into $4^m$ closed arcs of equal normalized arclength, and let
$h:\partial\mathbb D\to\Gamma_\rho$ be the positively oriented natural
parametrization. Thus the four arcs in $\mathcal J_1$ correspond to the
four sides of the initial square, and subsequent subdivisions correspond
to successive replacements. The construction in \cite[Section 4.3]{HM}
gives
\begin{equation}\label{diamunderh}
    \operatorname{diam}h(I)=\rho^{m-1},
    \quad \text{for}\,\,\, I\in\mathcal J_m,\quad m\ge1.
\end{equation}
Moreover, since $1/4<\rho<1/2$, the estimates in \cite[Section 4.3]{HM} show that $\Gamma_\rho$ is of bounded turning, i.e., there is a constant $B_\rho<\infty$ such
that, for any two points on $\Gamma_\rho$, the smaller of the diameters
of the two joining subarcs is at most $B_\rho$ times their Euclidean distance.

We claim that 
\begin{equation}\label{eq:example-holder-parametrization}
    c_\rho|\xi_1-\xi_2|^\beta
    \le |h(\xi_1)-h(\xi_2)|
    \le C_\rho|\xi_1-\xi_2|^\beta,
    \quad \xi_1,\xi_2\in\partial\mathbb D,
\end{equation}
for some constants $C_\rho>c_\rho>0$. Indeed, let $s>0$ be the normalized length of the shorter circular arc joining $\xi_1$ and $\xi_2$, and choose $m\ge1$ with $4^{-m}\le s<4^{-m+1}$.
The shorter joining arc is covered by at most five consecutive arcs
in $\mathcal J_m$, while each of the two joining arcs
contains an arc in $\mathcal J_{m+1}$.
Thus \eqref{diamunderh} and bounded turning give
$$
B_\rho^{-1}\rho^m
\le |h(\xi)-h(\zeta)|
\le 5\rho^{m-1}.
$$
Since $\rho=4^{-\beta}$ and
$s\le|\xi_1-\xi_2|\le 2\pi s$, this proves
\eqref{eq:example-holder-parametrization}.
Consequently, $h:(\partial\D,|\cdot|)\to(\Gamma_\rho,|\cdot|)$
is quasisymmetric, with quasisymmetry constant depending only on $\rho$.

Now we extend $h:\partial\mathbb D\to\Gamma_\rho$ to an orientation-preserving quasiconformal homeomorphism $\Phi:\mathbb C\to\mathbb C$. We first apply M\"obius transformations to send one
point on $\partial\D$ and its image to $\infty$. Then Tukia's extension theorem \cite{Tukia} gives an orientation-preserving quasiconformal extension
$\widetilde\Phi:\hat{\mathbb C}\to\hat{\mathbb C}$.
Because $h$ is positively oriented, $\widetilde\Phi$ maps
$\mathbb D$ onto the bounded component of
$\mathbb C\setminus\Gamma_\rho$. Thus
$a:=\widetilde\Phi^{-1}(\infty)$ lies in
$\hat{\mathbb C}\setminus\overline{\mathbb D}$.

Next we
choose an orientation-preserving smooth diffeomorphism $T:\hat{\mathbb C}\to\hat{\mathbb C}$ that is the identity on a neighborhood of
$\overline{\mathbb D}$ and satisfies $T(\infty)=a$.
Then $\Phi:=\widetilde\Phi\circ T$ satisfies $\Phi|_{\partial\D}=h$ and $\Phi(\infty)=\infty$, which gives us an orientation-preserving quasiconformal homeomorphism $\Phi:\mathbb C\to\mathbb C$. In particular, by \cite[Section 2.3]{tukia1980quasisymmetric}, $\Phi$ is $\eta$-quasisymmetric for some distortion
function $\eta:[0,\infty)\to [0,\infty)$.

\subsection{Length estimates for images of concentric circles}

For each $n\ge0$, let $N_n:=4^{n+1}$.
We denote the cyclically ordered endpoints of the arcs in
$\mathcal J_{n+1}$ by $\omega_{n,1},\ldots,\omega_{n,N_n}$. Set $x_{n,j}:=h(\omega_{n,j})$ for $1\le j\le N_n$, and adopt
the cyclic conventions
$\omega_{n,N_n+1}:=\omega_{n,1}$ and
$x_{n,N_n+1}:=x_{n,1}$.
Since the replacement procedure preserves the vertices, 
$$ |x_{n,j+1}-x_{n,j}|=\rho^n, 
    \quad
    |\omega_{n,j+1}-\omega_{n,j}|
    =2\sin\frac{\pi}{N_n}\ge \frac4{N_n}=4^{-n}.$$
Choose $0<\tau\le\frac{1}{2}$ such that $\eta(\tau)<\frac{1}{4}$.
If $0<r<1$ and $1-r\le\tau4^{-n}$, the quasisymmetry of $\Phi$ yields that
$$\frac{|\Phi(r\omega_{n,j})-x_{n,j}|}{\rho^n}
    \le
    \eta\!\left(
    \frac{1-r}{|\omega_{n,j+1}-\omega_{n,j}|}
    \right)
    \le\eta(\tau)<\frac14.$$
Applying this for consecutive vertices, by the triangle inequality, we obtain
$$  |\Phi(r\omega_{n,j+1})-\Phi(r\omega_{n,j})|
    \ge\frac12\rho^n.$$

Since $\Phi$ is a homeomorphism, the points $\Phi(r\omega_{n,j}), \,\,j=1,2,\cdots,N_n$, divide the Jordan curve $\Phi(r\partial\mathbb D)$
into subarcs that intersect only at their endpoints. Each subarc has
$\mathcal H^1$-measure at least the distance between its endpoints. Therefore,
\begin{equation}\label{H1lower}
    \mathcal H^1\bigl(\Phi(r\partial\mathbb D)\bigr)
    \ge\sum_{j=1}^{N_n}
    |\Phi(r\omega_{n,j+1})-\Phi(r\omega_{n,j})|
    \ge2(4\rho)^n.
\end{equation}
For any $r>0$ with $1-\tau<r<1$, choose $n_r\ge0$ such that
$$   \tau4^{-(n_r+1)}<1-r\le\tau4^{-n_r}.$$
Combing with \eqref{settingofcounter} and \eqref{H1lower}, we have
\begin{equation}\label{eq:weighted-example-length'}
    \mathcal H^1\bigl(\Phi(r\partial\mathbb D)\bigr)
    \ge 2\cdot4^{-\lambda}\tau^{\lambda}(1-r)^{-\lambda},
    \quad \text{for}\,\,\,1-\tau<r<1.
\end{equation}

\subsection{ Proof of Theorem \ref{thm:weighted-counterexample}.} Since $\Phi(\overline{\mathbb D})\subset {\mathbb C}$ is a compact set, we can postcompose $\Phi$ with an orientation-preserving similarity such that
$$ \Phi(0)=0
    \quad\text{and}\quad
    \Phi(\overline{\mathbb D})\Subset\mathbb D.$$
The similarity changes only the positive constant in
\eqref{eq:weighted-example-length'}. Retaining the notation $\Phi$, we define
$$ \Omega_\alpha:=\Phi^{-1}(\mathbb D)
    \quad\text{and}\quad
    f_\alpha:=\Phi^{-1}|_{\mathbb D}:\D\to \Omega_\alpha.$$
Then $f_\alpha$ is an
orientation-preserving quasiconformal homeomorphism with
$f_\alpha(0)=0$. It is also bounded because
$\Phi^{-1}(\overline{\mathbb D})$ is compact.
Moreover, for every $r>0$, we have
$$  F_r:=\{z\in\D:|f_\alpha(z)|=r\}
    =\D\cap\Phi(r\partial\mathbb D).$$
In particular, the assumption $ \Phi(\overline{\mathbb D})\Subset\mathbb D$ ensures that
\begin{equation}\label{eq:example-level-sets}
    F_r=\Phi(r\partial\mathbb D),\quad\text{for every}\,\, 0<r<1.
\end{equation}
The function $r\mapsto\mathcal H^1(F_r)$ is lower semicontinuous
and hence measurable.

Let $p>0$. Since $1-\tau\ge1/2$, the function
$r^{p-1}$ is bounded below by a positive constant on
$(1-\tau,1)$. Using \eqref{eq:weighted-example-length'},
\eqref{eq:example-level-sets}, and
$\lambda(\alpha+2)=1$, we conclude that
$$\begin{aligned}
    \int_0^\infty r^{p-1}\mathcal H^1(F_r)^{\alpha+2}\,dr
    &\ge C_{p,\alpha}\int_{1-\tau}^1(1-r)^{-\lambda ({\alpha+2})}\,dr\\
    &=C_{p,\alpha}\int_{1-\tau}^1\frac{dr}{1-r}
    =+\infty.
\end{aligned}$$

On the other hand, boundedness of $f_\alpha$ and
$\alpha>-1$ imply that $f_\alpha\in A_\alpha^p(\D)$. Since the construction is independent of $p$, the same mapping $f_\alpha$ works for every $p>0$.\qed

\begin{rem}
The restriction $\alpha>0$ arises naturally from the construction.
Indeed, the snowflake parameter $\rho=4^{\lambda-1}$ must satisfy
$1/4<\rho<1/2$, which is equivalent to $\alpha>0$.
\end{rem}

\begin{rem}\label{rem:K-blowup}
The maximal dilatations of the constructed maps satisfy
$$
K(\Phi)\ge K(f_\alpha)
\ge
\frac{\sqrt{\alpha+1}+1}{\sqrt{\alpha+1}-1}
\longrightarrow\infty
\quad\text{as }\alpha\downarrow0.
$$

Indeed, \eqref{eq:example-holder-parametrization} shows
that $h$ is bi-Lipschitz from
$(\partial\D,|\cdot|^\beta)$ onto
$(\Gamma_\rho,|\cdot|)$.
Since the subsequent similarity normalization preserves
Hausdorff dimension,
\begin{equation}\label{eq:snowflake-HD}
\dim_H\Phi(\partial\D)
=
\dim_H\Gamma_\rho
=
\frac1\beta
=
\frac{\alpha+2}{\alpha+1}.
\end{equation}

Recall that $\overline{\D}\Subset\Omega_\alpha$ and
$f_\alpha^{-1}=\Phi|_{\Omega_\alpha}$.
Extend the Beltrami coefficient of
$\Phi|_{\Omega_\alpha}$ by zero outside $\Omega_\alpha$.
The measurable Riemann mapping theorem gives a global
quasiconformal homeomorphism $\Psi:\mathbb C\to\mathbb C$ with the same Beltrami coefficient and
$$
K(\Psi)
=
K(\Phi|_{\Omega_\alpha})
=
K(f_\alpha).
$$
Since $\Phi$ and $\Psi$ have the same Beltrami coefficient
on $\Omega_\alpha$, the map
$H:=\Phi\circ\Psi^{-1}$ is conformal on
$\Psi(\Omega_\alpha)$.
In particular, it is locally bi-Lipschitz near the compact
curve $\Psi(\partial\D)$ and therefore preserves its
Hausdorff dimension.

Consequently, \eqref{eq:snowflake-HD} and Smirnov's
quasicircle dimension theorem
\cite[Theorem~1]{Smirnov} yield
$$
\frac{\alpha+2}{\alpha+1}
=
\dim_H\Phi(\partial\D)
=
\dim_H\Psi(\partial\D)
\le
1+
\left(
\frac{K(f_\alpha)-1}{K(f_\alpha)+1}
\right)^2.
$$
Rearranging gives the asserted lower bound for
$K(f_\alpha)$, and the bound for $K(\Phi)$ follows by
restriction and invariance of maximal dilatation under
inversion.

\end{rem}

\begin{rem}\label{rem:Hayman-Wu-and-Hausdorff-to-harmonic}
The failure of the Hayman--Wu bound in the quasiconformal setting was established by \cite[Theorem 1.1]{AstalaFernandezRohde1993}, where the counterexample yields an infinite sum of component diameters. Our snowflake construction exhibits a different way of the failure: the level sets are Jordan curves whose lengths are unbounded, although they are rectifiable for almost every radius.

Indeed, $\Phi(\partial \D)$ is a snowflake of Hausdorff dimension greater than $1$, hence $\HH^1(\Phi(\partial \D))=\infty$. As $r\uparrow 1$, the curves $F_r=\Phi(r\partial\D)=\partial \Phi(r\D)$ converge to this snowflake, this gives $\HH^1(F_r)\to\infty$, whereas $\HH^1(F_r)<\infty$ for almost every $0<r<1$.
Thus, no finite uniform length bound exists even for this single mapping.

Moreover, since $\Phi(0)=0$ and $ \Phi(\overline{\mathbb D})\Subset\mathbb D$, the curve $F_r=\partial\Phi(r\D)$
surrounds $0$ and separates it from $\partial\D$ for every $0<r<1$. 
Therefore, we have $\omega_\D(0,F_r)=1$. The divergence of $\HH^1(F_r)$ also shows the failure of 
\eqref{eq:Hausdorff-to-harmonic}, even with a 
constant depending on the mapping.
\end{rem}

	\bibliographystyle{amsalpha}
\bibliography{bibliography}

\end{document}